\documentclass[11pt,draft]{article}
\usepackage{amsfonts}
\usepackage{amsmath}
\usepackage{a4wide}

\newcommand{\R}{\mathbb{R}}

\newtheorem{theorem}{Theorem}

\newtheorem{algorithm}{Algorithm}
\newtheorem{corollary}{Corollary}
\newtheorem{definition}{Definition}
\newtheorem{example}{Example}

\newtheorem{proposition}{Proposition}
\newenvironment{proof}{\noindent\textit{Proof:}\ }{\hfill $\Box$\newline}

\begin{document}

\title{Extending Translating Solitons in Semi-Riemannian Manifolds}

\author{
Erdem Kocaku\c{s}akl\i \\
Department of Mathematics, Faculty of Science, University \\
of Ankara Tandogan, Ankara, TURKEY\\
kocakusakli@ankara.edu.tr \and 
Miguel Ortega \\
Institute of Mathematics, Department of Geometry and Topology,\\
University of Granada, Granada, SPAIN\\
miortega@ugr.es
}
\date{\today}
\maketitle


\begin{abstract}
In this paper, we recall some general properties and theorems about
Translating Solitons in Semi Riemannian Manifolds. Moreover, we investigate
those which are invariant by the action of a Lie group of
isometries of the ambient space, by paying attention to the behaviour close to the singular orbit (if any) and at infinity. Then, we provide some related examples.
\end{abstract}

\section{Introduction}
Given a smooth manifold $M$, assume a family of smooth immersions in a semi-Rieman\-nian manifold $(\mathbf{M},\mathbf{g})$, $F_t:M\rightarrow \mathbf{M}$, $t\in [0,\delta)$, $\delta>0$, with mean curvature vector $\vec{H}_t$. The initial immersion $F_0$ is called a solution to the \textit{mean curvature flow} (up to local diffeomorphism) if 
\begin{equation}\label{MCF} \left(\frac{d}{dt} F_t\right)^{\perp} = \vec{H}_t,
\end{equation}
where $\perp$ means the orthogonal projection on the normal bundle. In the Euclidean and Minkowski space, there is a famous family of such immersions, namely, translating solitons. A submanifold is called \textit{translating soliton} in the Euclidean Space when its mean curvature $\vec{H}$ satisfies the following equation:
\begin{equation} \label{soliton}
\vec{H} = v^{\perp},
\end{equation}
for some constant unit vector $v\in\R^{n+1}$. Indeed, if a submanifold $F:M\rightarrow \mathbb{R}^{n+1}$ satisfies this condition, then it is possible to define the forever flow $\Gamma:M\times[0,+\infty)\rightarrow \mathbb{R}^{n+1}$, $\Gamma(p,t)=F_t(p)=F(p)+tv$. Clearly, 
\[ \left(\frac{d}{dt} F_t\right)^{\perp} = v^{\perp}=\vec{H}.
\]
This justifies our definition. Same situation holds in Minkowski Space. Until now such solutions have been almost exclusively studied in the case where the ambient space is the Euclidean (or the Minkowski) space. For a good list of known examples, see \cite{Martin}. Probably, the most famous examples are the Grim Reaper curve in $\R^2$ and the translating paraboloid and translating catenoid, \cite{Clut}. Also, in \cite{N} there are some examples with complicated topology. Recently, in \cite{Li}, the authors studied those translating solitons in Minkowski 3-space with rotational  symmetry.

If one wants to generalize \eqref{soliton}, the simplest way is to choose a parallel vector field. But manifolds admitting such a vector field are locally a product $M\times\R$. Thus, in \cite{ortega}, the authors introduce the notion of (graphical) translating solitons on a semi-Riemannian product $M\times \R$. Needles to say, their study include the Riemannian case. When the translating soliton is the graph of map $u$ defined on (an open subset of) $M$, the corresponding partial differential equation that $u$ must satisfy is obtained in \cite{ortega}. Alghouth this paper includes more results, one of the main concern is the study of translating solitons which are invariant under the action of a Lie group by isometries on $M$. This action is very easily extended to $M\times\R$. The authors focused on the case when the quotient map is an open interval, $M/\Sigma\equiv I\subset\R$. This is so because among the classical examples, the translating paraboloid and the translating catenoid are constructed this way in \cite{Clut}. 

In this paper, we would like to continue the study of \cite{ortega} by further developing some ideas. 

Firstly, in the Preliminaries Section we recall some known results that we will use later. Among them, we include a summary of  \cite{ortega}, where we find the first steps of translating solitons in product spaces $M\times\R$ with a product metric $g_M+\varepsilon dt^2$, with $\varepsilon=\pm 1$ and $g_M$ the metric on $M$. Since the manifold $M$ might not be complete, we can almost say that we are dealing with a \textit{semi-Riemannian cohomogeneity of degree one  $\Sigma$-manifold}, since there is a Lie group acting by isometries and the quotient is a 1-dimensional manifold  (see \cite{Ale-Ale}.) In this setting, we can construct our translating solitons from the solutions to an ODE. This is clarified in the new Algorithm \ref{method}, which was not included in \cite{ortega}. 

Section \ref{boundaryproblem} is devoted to studying the already mentioned ODE, from two points of view. One of them is solving a boundary problem. Indeed, given $h\in C^1(a,b)$ such that $\lim\limits_{s\to a}h(s)=+\infty$, and $\varepsilon,\tilde{\varepsilon}\in\{1,-1\}$, 
consider
\[
w'(s) =(\tilde{\varepsilon}+\varepsilon w^{2}(s))(1-w(s)h(s)), \quad 
w(a) =0. 
\]
We show in Theorem \ref{solution} that there exists a solution under a not very restrictive condition on function $h$. The reason to consider this problem is the following. In the Euclidean Space $\R^{n+1}$, graphical translating solitons which are invariant by $SO(n)$ and touching the axis or rotation (in other words, \textit{rotationally invariant}) arise from the solution of the following boundary problem:
\[ w'(s) =\big(1+w^2(s)\big)\Big(1-\frac{n-1}{s}w(s)\Big), \quad w(0)=0.\]
In fact, the solution gives rise to the famous example known as  Translating Paraboloid. Clearly, function  $h:(0,+\infty)\rightarrow\R$, $h(s)=(n-1)/s$, so that we are studying  a much more general problem by choosing any $C^1$ function $h$ satisfying simple conditions on the boundary. 

The existence of solutions in Theorem \ref{solution} is just local, i.~e., in a small interval $[a,a+\delta)$. Thus, the second point of view consist of the extension of our solutions. In this way, in Propositions \ref{extending+1+1} and \ref{extending-1-1} of Section \ref{extensions}, we show that for $\varepsilon\tilde{\varepsilon}=1$ and $h>0$ or $h<0$, it is possible to extend the solution to the interval $[s_0,b)$, where $s_0\in (a,b)$ is the chosen initial point. In Proposition \ref{extending+}, we show some reasonable conditions under which, the solutions defined on $[s_0,b)$ admit $\lim\limits_{s\to b}w(s)\in\R$. 

In Section \ref{singularities} we pay attention to manifods admiting a Lie group $\Sigma$ acting by isometries, such that the orbits are (CMC) hypersurfaces, the quotient manifold is an  interval $[a,b)$, and the mean curvature of the orbits tend to infinity when approaching the singular orbit. For example, this is the case of the Euclidean Space $\R^n$ under the action of $SO(n)$. Then, we apply our previous  computations to obtain solutions (denoted by $w$) to the corresponding boundary problem. Next, we use Algorithm \ref{method} to primitives of them, namely $f=\int w$, to obtain $\Sigma$-invariant translating solitons.

Last, but not least, we show some examples in Section \ref{ejemplos}. On one hand, we exhibit translating solitons in $\mathbb{H}^n\times\R$ whose invariant subsets are horospheres. Except one case, all of them are not entire, in the sense that they admit finite time blow-ups. Also, we make a study on the round sphere, where we obtain translating solitons defined on the whole sphere but removing one or two points. 

\section{Preliminaries}

The following results can be found in \cite{ortega}. Assume that $(M,g)$ is a connected semi-Riemannian manifold of dimension $n\geq 2$ and index $0\leq \alpha \leq n-1$.  Given $\varepsilon =\pm 1$, we construct the semi-Riemannian product $\widetilde{M}=M\times \mathbb{R}$ with metric $\left\langle ,\right\rangle =g+\varepsilon dt^{2}$. The vector
field $\partial _{t}\in \chi (\widetilde{M})$ is obviously Killing and unit,
spacelike when $\varepsilon =+1$ and timelike when $\varepsilon =-1$.  Now
let $F:\Gamma \rightarrow \widetilde{M}$ be a submanifold with mean
curvature vector $\vec{H}$.  Denote by $\partial _{t}^{\perp }$
the normal component of $\partial _{t}$ along $F$.
\renewcommand{\thedefinition}{\Alph{definition}}
\begin{definition}
With the previous notation, we will call F a (vertical) translaing soliton
of mean curvature flow, or simply, a translating soliton, if $\vec{H}=\partial _{t}^{\perp }$.
\end{definition}

In this paper, we will focus on graphical translating solitons. Namely,
given $u\in $ $C^{2}(M)$, we construct its graph map $F:M\rightarrow M\times 
\mathbb{R=}:\widetilde{M}$, $F(x)=(x,u(x))$.  Let $\nu$ be the upward normal
vector along $F$ with $\varepsilon '=\mathrm{sign}(\left\langle \nu ,\nu
\right\rangle )=\pm 1$.  
\renewcommand{\theproposition}{\Alph{proposition}}

Let $\Sigma $ be a Lie
group acting by isometries on $M$ and $\pi :M\rightarrow I$ be a
submersion, $I$ and open interval, such that the fibers of $\pi $ are orbits of the action. In addition, assume that $\pi$ is a semi-Riemannian submersion with
constant mean curvature fibers. For each $s\in I$, $\pi^{-1}\{s\}\cong \Sigma$ is a hypersurface with constant mean curvature. The value of the mean curvature of $\pi^{-1}\{s\}$ is denoted by $h(s)$. Then, we have a function $h:I\rightarrow\mathbb{R}$. We will say that \textit{function $h$ represents the mean curvature of the orbits.} Given a map $F:M\rightarrow P$, where $P$ is another set, is $\Sigma$-invariant when $F(\sigma\cdot x)=F(x)$ for any $\sigma\in\Sigma$ and any $x\in M$. 

\renewcommand{\thetheorem}{\Alph{theorem}}
\begin{theorem}
Let $(M,g)$ be a connected semi-Riemannian manifold. Let $\Sigma $ be a Lie
group acting by isometries on $M$ and $\pi : (M,g_{M})\rightarrow (I,\widetilde{\varepsilon }ds^{2})$ be a semi-Riemannian submersion, $I$ an open interval, such that
the fibers of $\pi $ are orbits of the action, with function $h$
representing the mean curvature of the orbits. Take $u\in C^{2}(M,\mathbb{R}) $ and consider its graph map 
\[
F:M\rightarrow M\times \mathbb{R},F(x)=(x,u(x))
\]
for any $x\in M$. Then, $F$ is a $\Sigma $-invariant translating soliton if,
and only if, there exists a solution $f\in C^{2}(I,\mathbb{R})$ to 
\begin{equation}
f^{''}(s)=(\widetilde{\varepsilon }+\varepsilon (f^{\prime
})^{2}(s))(1-h(s)f'(s))  \label{***}
\end{equation}
such that $u=f\circ \pi $.
\end{theorem}

\renewcommand{\thecorollary}{\Alph{corollary}}

The following results study some conditions to obtain translating solitons which cannot be globally defined, in the sense that they are not entire graphs. Instead, they are defined on some smaller subsets, and converging to infinity. 
\begin{corollary}
Under the same conditions, assume that $\varepsilon \widetilde{\varepsilon}=-1$. Then, given $s_{0}\in I,f_{1}\in (-1,1)$ and $f_{o}\in \mathbb{R}$, there exists a solution $f:I\rightarrow \mathbb{R}$ to \eqref{***} such that $f(s_{o})=f_{0}$ and $f'(s_o)$ $=f_{1}$.
\end{corollary}
 
\begin{corollary}
Let $\varepsilon =1$ and $\widetilde{\varepsilon }=-1$.  Take $c\in I$. 
Consider any $\lambda >1$.
\begin{enumerate}
\item Let $f$ be a solution \eqref{***} such that $f'(c)=\lambda$.   If $h(s)\leq 0$ for any $s\geq c,s\in I$, and $\mathrm{sup}(I)>c+\coth^{-1}(\lambda )=a$, then $f$ admits a finite time blow up before $a$.

\item Let $f$ be a solution \eqref{***} such that 
$f'(c)=-\lambda$.   If $h(s)\leq -1$ for any $s\geq c,s\in I$, and $\mathrm{sup}(I)>c+\dfrac{\coth ^{-1}(\lambda )}{\lambda -1}=a$, then $f$ admits a finite time blow up before $a$.
\end{enumerate}
\end{corollary}

\begin{corollary}
Let $\varepsilon =\widetilde{\varepsilon }=-1\ $and $c\in I$.Consider any $%
\lambda >0$ such that $\mathrm{sup}(I)>c+\dfrac{1}{\lambda }=a$.
\begin{enumerate}
\item Let $f$ be a solution \eqref{***} such that $f'(c)=-\lambda$.   If $h(s)>0$ for any $s\geq c,s\in I$, then $f$ admits a
finite time blow up before $a$.

\item Let $f$ be a solution \eqref{***} such that $f'(c)=\lambda$.   If $h(s)>\dfrac{1}{f'(c)}$ for any $s\geq c,s\in I$ then $f$ admits a finite time blow up before $a$.
\end{enumerate}
\end{corollary}

\begin{corollary}
\begin{enumerate}
\item If  $\varepsilon =\widetilde{\varepsilon }=1$, consider any $\lambda >0$ such that $\mathrm{sup}(I)>c+\dfrac{1}{\lambda }$.  Take any solution $f$ to 
\eqref{***}.   If either $f'(c)=\lambda $ and $h(s)<0$
for any $s\geq c,s\in I$, or $f'(c)=-\lambda $ and $h(s)<\dfrac{1}{f'(c)}$ for any $s\geq c,s\in I$, then $f$ admits a finite time blow up before $a$.

\item If $\varepsilon =-1$ and $\widetilde{\varepsilon }=1$, consider any $\lambda >1$. Let $f$ be a solution to \eqref{***} such that $f'(c)=-\lambda$.   If $h(s)>0$ for any $s\geq c,s\in I$, and $\mathrm{sup}(I)>c+\coth ^{-1}(\lambda )=a$, then $f$ admits a finite time blow up before $a$.

\item If $\varepsilon =-1$ and $\widetilde{\varepsilon }=1$, consider any $\lambda >1$. Let $f$ be a solution to \eqref{***} such that $f'(c)=\lambda$.   If $h(s)>1$ for any $s\geq c,s\in I$, and $\mathrm{sup}(I)>c+\dfrac{\coth ^{-1}(\lambda )}{\lambda -1}=a$, then $f$ admits a finite time blow up before $a$.
\end{enumerate}
\end{corollary}

Until now, we are recalling known results. But for the sake of clarity, we now introduce a method to construct a translating soliton in a manifold foliated by the orbits of the action of a Lie group acting by isometries.
\begin{algorithm}\label{method} 
\normalfont Let $(M,g)$ be semi-Riemannian manifold, $\Sigma $ a Lie subgroup of $\mathrm{Iso}(M,g)$, and $I$ open interval. Choose $\varepsilon \in \left\{ \pm 1\right\}$.  The metric in $M\times\R$ is $\left\langle ,\right\rangle =g+dt^{2}$.
\end{algorithm}

\begin{enumerate}
\item Assume $\phi:M\rightarrow (\Sigma/K)\times I$ is a diffeomorphism, for some subgroup $K$, such that its restriction $\pi :M\rightarrow I$ satisfies $\left\vert \nabla \pi \right\vert ^{2}\neq 0$.

\item By a change of variable, recompute $\pi $ $($and $\phi )$ to obtain $\left\vert \nabla \pi \right\vert ^{2}=\tilde{\varepsilon }=\pm 1$.

\item For each $s\in I$ compute the mean curvature $h(s)$ of the fiber $\pi
^{-1}\left\{ s\right\} \subset M$.  Note $\pi ^{-1}\left\{ s\right\} \cong
\Sigma/K$.

\item Solve the following problem for some initial values in an interval $J\subset I$, 
\[
f^{''}(s)=(\widetilde{\varepsilon }+\varepsilon (f^{\prime
}(s))^{2})(1-f'(s)h(s)).
\]

\item The translating soliton can be constructed by one of the following equivalent ways:
\begin{align*}
F :(\Sigma/K) \times J\rightarrow M\times \mathbb{R},\ F(\sigma
,s)=(\phi ^{-1}(\sigma ,s),f(s)). \\
\overline{F} :\phi ^{-1}((\Sigma/K) \times J)\rightarrow M\times \mathbb{R},\ 
\overline{F}(x)=(x,f(\pi (x)).
\end{align*}
\end{enumerate}

We will use the following tools in order to solve our ODE with singularities. See \cite{Wiggins} for details. We consider the following linear ODE, 
\begin{equation}
\dot{y}=Ay,\text{ }y\in \mathbb{R}^{n}  \label{4*}
\end{equation}%
From elementary linear algebra we can find a linear transformation $T$ which
transforms the linear equation \eqref{4*} into block diagonal
form
\begin{equation}
\begin{bmatrix}
\dot{u} \\ 
\dot{v} \\ 
\dot{w}
\end{bmatrix}%
=%
\begin{bmatrix}
A_{s} & 0 & 0 \\ 
0 & A_{u} & 0 \\ 
0 & 0 & A_{c}%
\end{bmatrix}%
\begin{bmatrix}
u \\ 
v \\ 
w%
\end{bmatrix}
\label{1*}
\end{equation}%
where $T^{-1}y=(u,v,w)\in R_{s}\times $ $R_{u}\times R_{c}$, $s+u+c=n$, $%
A_{s}$ is an $s\times s$ matrix having eigenvalues with negative real part, $%
A_{u}$ is an $u\times u$ matrix having eigenvalues with positive real part,
and $A_{c}$ is an $c\times c$ matrix having eigenvalues with zero real part.
Moreover, we know that%
\begin{eqnarray}
\dot{u} &=&A_{s}u+R_{s}(u,v,w)  \notag \\
\dot{v} &=&A_{u}v+R_{u}(u,v,w)  \label{2*} \\
\dot{w} &=&A_{c}w+R_{c}(u,v,w)  \notag
\end{eqnarray}%
where $R_{s}(u,v,w),R_{u}(u,v,w)$ and $R_{c}(u,v,w)$are the first $s,u$ and $%
c$ components, respectively, of the vector $T^{-1}R(T^{y})$.

\begin{theorem}\label{subs}
Suppose \eqref{2*} is $C^{r}$, $r\geq 2$. Then the fixed
point $(u,v,w)=0$ of \eqref{2*} possesses a $C^{r}$ $s-$%
dimensional local, invariant stable manifold, $W_{loc}^{s}$ $(0)$, a $C^{r}$ 
$u-$dimensional local, invariant unstable manifold, $W_{loc}^{u}$ $(0)$ and
a $C^{r}$ $c-$dimensional local, invariant center manifold $W_{loc}^{c}(0)$,
all intersecting at $(u,v,w)=0$. These manifolds are all tangent to the
respective invariant subspaces of the linear vector field \eqref{1*} at the origin and, hence, are locally representable as graphs. In particular, we have 
\[
W_{loc}^{s}(0)=\left\{ 
\begin{array}{c}
(u,v,w)\in \mathbb{R}^{s}\times \mathbb{R}^{u}\times \mathbb{R}%
^{c}\left\vert {}\right. \text{ }v=h_{v}^{s}(u),w=h_{w}^{s}(u); \\ 
\text{ }Dh_{v}^{s}(0)=0,Dh_{w}^{s}(0)=0;\text{ }\left\vert u\right\vert 
\text{ sufficiently small}%
\end{array}%
\right\}
\]
\bigskip 
\[
W_{loc}^{u}(0)=\left\{ 
\begin{array}{c}
(u,v,w)\in \mathbb{R}^{s}\times \mathbb{R}^{u}\times \mathbb{R}%
^{c}\left\vert {}\right. \text{ }u=h_{u}^{u}(v),w=h_{w}^{u}(v); \\ 
\text{ }Dh_{u}^{u}(0)=0,Dh_{w}^{u}(0)=0;\text{ }\left\vert v\right\vert 
\text{ sufficiently small}%
\end{array}%
\right\}
\]
\bigskip 
\[
W_{loc}^{c}(0)=\left\{ 
\begin{array}{c}
(u,v,w)\in \mathbb{R}^{s}\times \mathbb{R}^{u}\times \mathbb{R}%
^{c}\left\vert {}\right. \text{ }u=h_{u}^{c}(w),v=h_{v}^{c}(w); \\ 
\text{ }Dh_{u}^{c}(0)=0,Dh_{v}^{c}(0)=0;\text{ }\left\vert w\right\vert 
\text{ sufficiently small}
\end{array}
\right\}
\]
where $h_{v}^{s}(u)$, $h_{w}^{s}(u)$, $h_{u}^{u}(v),h_{w}^{u}(v);$, 
$h_{u}^{c}(w)$ and $h_{v}^{c}(w)$ are $C^{r}$ functions. Moreover,
trajectories in $W_{loc}^{s}$ $(0)$ and $W_{loc}^{u}$ $(0)$ have the same
asymptotic properties as trajectories in $E^{s}$ and $E^{u}$, respectively.
\end{theorem}

\section{Solution to a Boundary Problem With Singularity}\label{boundaryproblem}
\setcounter{theorem}{0}
\renewcommand{\thetheorem}{\arabic{theorem}}
\setcounter{corollary}{0}
\renewcommand{\thecorollary}{\arabic{corollary}}

\begin{theorem}\label{solution}
Given $a\in\R$, $b\leq +\infty$, $\varepsilon,\tilde{\varepsilon}\in\{1,-1\}$, choose $q\in C^{1}[a,b)$ such that   $q(a)=0$, $q(s)\neq 0$ for any $s>a$, $\tilde \varepsilon q^{\prime}(a)\geq 0$, and define $h:(a,b)\rightarrow \mathbb{R}$
given by $h=1/q$. Then, the boundary problem 
\begin{equation} \label{bproblem}
w'(s) =(\tilde{\varepsilon}+\varepsilon w^{2}(s))(1-w(s)h(s)), \quad 
w(a) =0 
\end{equation}
has a solution $w:\left[ a,a+\delta \right) \rightarrow \mathbb{R}$
for a suitable small $\delta >0$. 
\end{theorem}

\begin{proof} It is well-known that it is possible to extend $q$ a little in the following way. For some $\rho>0$, there exist a (non unique) $\tilde{q}\in C^1 (a-\rho,b)$ such that $\tilde{q}(s)=q(s)$ for any $s\geq a$. Then, we simply work on the interval $I=(a-\rho,b)$. The extension is not going to be crutial, because we really just care for $s\in I$, $s\geq a$. We consider the following autonomous vector field:
\[
X:I\times \mathbb{R}\rightarrow \mathbb{R}^{2},\ 
X(s,x)=\big(q(s),(\tilde{\varepsilon}+\varepsilon x^{2})(q(s)-x)\big).
\]
Note that $X(a,0)=(0,0)$. Moreover, at $(a,0)$ the linearlization is 
\[
DX(a,0)=\left[ 
\begin{array}{cc}
q'(a) & 0 \\ 
\tilde\varepsilon q'(a) & -\tilde{\varepsilon}
\end{array}
\right] .
\]
Since $\tilde \varepsilon q'(a)\geq 0$, there are two eigenvalues $\lambda _{1}=q'(a)$ and $\lambda _{2}=-\tilde{\varepsilon}$, with different sign or $q'(a)=0$, with corresponding eigenvectors 
$v_{1}=(\tilde\varepsilon +q'(a),\tilde{\varepsilon}q^{\prime}(a))^{t}$, $v_{2}=(0,1)^{t}$. 
By Theorem \ref{subs}, there exist a 1-dimensional manifold (of fixed point), around $(a,0)$, whose tangent space at $(a,0)$ is spanned by $v_{1}$, which is a graph in a small interval around $s_{0}$, namely
\[
W=\left\{ (s,x)\in I\times \mathbb{R}:x=w(s),\text{ \ }\left\vert
s-a\right\vert <\delta \right\} 
\]
for some function $w$ defined on a small interval $(-\delta ,\delta)$.  This means that our dynamical system has a solution 
\[
\alpha :(-\delta ,\delta )\rightarrow W,\text{ \ }\alpha (t)=(s(t),x(t)), 
\text{ with }\alpha '(t)=X(\alpha (t))
\]
such that $\alpha (0)=(a,0)$, $\alpha'(0)=\lambda v_1$ for some $\lambda \in \mathbb{R}$, $\lambda \neq 0$, and $x(t)=w(s(t))$.  We compose with the inverse of $s$, so that $w(s)=x(t(s))$. Moreover, since $X(\alpha (t))=\alpha '(t)$, we have $\alpha'(t)=(s'(t),x'(t)) = \big(q(s(t)),(\tilde{\varepsilon}+\varepsilon x(t)^{2})(q(s(t))-x(t)\big)$, and so for $s>a$, 
\begin{eqnarray*}
w'(s) &=&x'(t(s))t'(s)=\frac{x'(t(s))}{ s'(t) } = 
\frac{
(\tilde{\varepsilon}+\varepsilon x(t(s))^{2}) (q(s)-x(t(s)))}{q(s)} \\
&=&(\tilde{\varepsilon}+\varepsilon w(s)^{2})(1-h(s)w(s)).
\end{eqnarray*}
According to \cite[p. 35]{Wiggins}, when $q'(a)\neq 0$, there is also uniqueness of solution. When $q'(a)=0$, solutions \textit{only differ by exponentially small functions of the distance from the fixed
point.}
\end{proof}

\section{Extension of Solutions}\label{extensions}

\renewcommand{\theproposition}{\arabic{proposition}}
\setcounter{proposition}{0}

Along this section, we will always assume the following:\\

\noindent{(H)} \textit{Given $a<b\leq +\infty$, take $s_0\in (a,b)$. Consider $h\in C^1(a,b)$ such that $h>0$. }
\begin{proposition}\label{extending+1+1} Assume (H).
\begin{enumerate}
\item For each $w_{0}\in 
\mathbb{R}$, the initial value problem 
\begin{equation}\label{problem-ext}
w'(s) =(1+w^{2}(s))(1-h(s)w(s)), \quad w(s_{0}) =w_{0}, 
\end{equation}
has a unique $C^{2}$-solution $w$ on $(s_{0}-\rho,b)$, for some $\rho>0$. 
\item If $b=+\infty$, then $\lim\limits_{s\rightarrow b}h(s)w(s)=1$.
\end{enumerate}
\end{proposition}
\begin{proof} First of all, there exists $\rho>0$ and $w:(s_0-\rho, s_{0}+\rho) \rightarrow \mathbb{R}$ a solution, and we wish  to extend it to $b$. In this proof, we will use the classical result of extension of solutions, which can be found for example on \cite[p 15]{Cod}, without saying explicitly. Note that function $F:(a,b)\times\R$, $F(s,x)=(1+x^2)(1-h(s)x)$, is continuous, so it will be bounded on compact   domains. Thus, we call 
\[
J=\left\{ s\in[s_0,b) : \text{ there exist }w:[s_{0},s) \rightarrow \mathbb{R}\right\}.
\]
Note that $[s_0,s_0+\delta)\subset J$. We want to show that $supJ=b$, so we take $s_1\in J$ such that $s_0<s_1<b$. 

If $w(s_{1})\leq \frac{1}{h(s_{1})}$, then $w'(s_{1})\geq 0$. There exists $\delta_1>0$ such that $a<s_1-\delta_1<s_1+\delta_1<b$ and $(s_1,w(s_1))$ is an interior point of $[s_1-\delta_1,s_1+\delta_1] \times [w(s_1)-1,1+1/h(s_1)]$, where $F$ is bounded. Thus, we can extend  $w$ a little. 

If $w(s_{1})>\frac{1}{h(s_{1})}>0$, then $w'(s_{1})<0$. There exists $\delta_1>0$ such that $a<s_1-\delta_1<s_1+\delta_1<b$ and $(s_1,w(s_1))$ is an interior point of $[s_1-\delta_1,s_1+\delta_1] \times [0,1+w(s_1)]$, where $F$ is bounded. Thus, we can extend  $w$ a little. 

Note also that if for some $s_1\in J$, $w(s_1)\geq 0$, then $w(s)\geq 0$ for any $s>s_1$, $s\in J$. Indeed, if for some $s_2>s_1$, $w(s_2)<0$, by the continuity of $w$ and $w'$, there exists $s_3\in (s_1,s_2)$ such that $w'(s_3)<0$ and $w(s_3)<0$. But by \eqref{problem-ext}, $w'(s_3)>0$, which is a contradiction.

Now, we call $\tilde{s}=\mathrm{sup}(J)\leq b$. Assume that $\tilde{s}<b$. Firstly, if $w$ is bounded on a small interval $[\tilde{s}-\delta_1,\tilde{s}]$, by our previous computations, we can extend $w$ a little to $[s_0,\tilde{s}+\delta)$, which is a contradition. Then, $w$ cannot be bounded when $s$ approaches $\tilde{s}$. Thus, there are two possibilities:\\
(a) There is a sequence $s_n\rightarrow\tilde{s}$ such that  $w(s_n)\nearrow +\infty$. For some $m$ natural number $w(s_n)\geq M$ for any $n\geq m$. By \eqref{problem-ext}, $w'(s_m)<0$ and $w'(s_{m+1})<0$. Since $w(s_m)<w(s_{m+1})$, then $w$ attains its maximum on $[s_m,s_{m+1}]$ at a point $t\in (s_m,s_{m+1})$. But then, $w'(t)=0$, and again by \eqref{problem-ext}, $w(t)=1/h(t)<M\leq w(s_m)<w(t)$. This is a contradiction.\\
(b) There is a sequence $s_n\rightarrow\tilde{s}$ such that  $w(s_n)\searrow -\infty$.  For some $m$ natural number $w(s_n)\leq -1$ for any $n\geq m$. By \eqref{problem-ext}, $w'(s_m)>0$ and $w'(s_{m+1})>0$. Since $w(s_m)>w(s_{m+1})$, then $w$ attains its minimum on $[s_m,s_{m+1}]$ at a point $t\in (s_m,s_{m+1})$. But then, $w'(t)=0$, and again by \eqref{problem-ext}, $w(t)=1/h(t)> w(s_{m+1})>w(t)$. This is a contradiction.

Therefore, the only possibility is $\mathrm{sup}(J)=b$. \\

Next, we want to study the behaviour of $w$ when $b=+\infty$. 

\noindent\textit{Case A:} Assume that there exist $s_{1}\geq s_0$ and $M>0$ such that for each $s\geq s_{1}$, $1-h(s)w(s)\geq M$. Then, for each $s\geq s_1$, $\frac{w'(s)}{1+w^2(s)}\geq M$. If we integrate this inequality, we have $\arctan(w(s))-\arctan(w(s_1))\geq M(s-s_1)$. Hence, 
for each $s\geq s_1$, we know $w(s)\geq \tan(Ms-Ms_1+\arctan(w(s_1)))$. Clearly, there exist $s$ big enough and $m$ a natural number such that $(2m+1)\pi/2 = Ms-Ms_1+\arctan(w(s_1)).$ This is a contradiction.\\

\noindent\textit{Case B:} Assume that there exist $s_{1}>s_0$ and $M<0$ such that for each $s\geq s_{1}$, $1-h(s)w(s)\leq M<0$. If we change $v=-w$, then $v'(s)=-w'(s)$ and $v'(s)=(1+v^{2})(-1-h(s)v(s))$. On the other hand, 
$1-h(s)w(s)=1+h(s)v(s)\leq M\leq 0$, and therefore $\frac{v'(s)}{1+v^2(s)}=-1-h(s)v(s))\geq -M>0$, for any $s\geq s_1$. Next, we repeat the steps of case A. 
\end{proof}

By Theorem  \ref{solution} and Proposition \ref{extending+1+1}, we obtain the following result. 
\begin{corollary} Assume (H), and in addition $\lim\limits_{s\to a}h(s)=+\infty$ and $\lim\limits_{s\to a} \frac{h'(s)}{h^2(s)}=h_1>0$. Then, the boundary problem \eqref{bproblem} has a unique globally defined solution $w\in C^1[a,b)$. 
\end{corollary}

\begin{proposition}\label{extending+} Assuming (H), suppose in addition $b<+\infty $ and there exist  $\lim\limits_{s\to b}h(s)=+\infty$.
\begin{enumerate}
\item If there exists $\lim\limits_{s\to b} \frac{h'(s)}{h^2(s)}= h_1\in [0,+\infty)$, there is a solution to \eqref{problem-ext} for certain $w_0$ such that  $\lim\limits_{s\to b} w(s)=0$ and $\lim\limits_{s\to b} w'(s)=\frac{1}{1+h_1}$. 
\item  If for some $M>0$ and $s_{1}\in \left[ s_{0},b\right)$, it holds $w(s)\geq M$ for every $s\geq s_{1}$,  then there exist $\lim\limits_{s\rightarrow b}w(s)=w_{1}\geq M\ $and $\lim\limits_{s\rightarrow b}w'(s)=-\infty$.
\item If there exist $M<0,$ $s_{1}\in \left[s_{0},b\right) $ such that for every $s\geq s_{1},$ $w(s)\leq M,$ then there exist $\lim\limits_{s\rightarrow b}w(s)=w_{1}\leq M\ $and $\lim\limits_{s\rightarrow b}w'(s)=+\infty.$
\end{enumerate}
\end{proposition}
\begin{proof} We check item 1. We consider the map $\phi:[0,b-s_0] \rightarrow [s_0,b]$ given by $\phi(u)=b-u$, and the function $\tilde{h}:(0,b-s_0]\rightarrow\R$, $\tilde{h}(u)=h(\phi(u))$. Note that $\phi$ also can be seen $\phi:[s_0,b]\rightarrow[0,b-s_0]$. 
Clearly, $\tilde{h}>0$, $\lim\limits_{u\to 0}\tilde{h}(u)=\lim\limits_{s\to b} h(s)=+\infty$, so we can define the function $q:[0,b-s_0]\rightarrow\R$, $q(u)=1/\tilde{h}(u)$ when $u>0$ and $q(0)=0$. Moreover, 
\[   \lim\limits_{u\to 0} q'(u)=\lim\limits_{u\to 0} \frac{h'(\phi(u))}{h^2(\phi(u))} = h_1\geq 0. 
\]
Thus, $q\in C^1[0,b-s_0)$. By Theorem \ref{extending+1+1}, but using $\tilde{h}$ instead of $h$, 
there is a solution $z:[0,b-s_0)\rightarrow\R$ to problem  \eqref{problem-ext}, such that $z(0)=0$. We define now the function $w:[s_0,b]\rightarrow \R$ given by $w(s)=-z(\phi(u))$. In particular, $\lim\limits_{s\to b} w(s)=0$ and $w'(s)=z'(\phi(u))=(1+w^2(s))(1-h(s)w(s))$. Moreover, $\lim\limits_{s\to b}w'(s)=\lim\limits_{u\to 0}z'(u)$. Note that
\begin{align*}
& \lim_{u\to 0} z'(u)\left( 1-\frac{1}{ \frac{\tilde{h}'(u)}{\tilde{h}^2(u)} }
\right) = \lim_{u\to 0}z'(u) -\lim_{u\to 0} \frac{z'(u)}{\frac{\tilde{h}'(u)}{\tilde{h}^2(u)}} = \lim_{u\to 0}z'(u) +\lim_{u\to 0}\frac{z(u)}{1/\tilde{h}(u)} \\
&= \lim_{u\to 0}\left(  (1+z^2(u))(1-\tilde{h}(u)z(u)) +\tilde{h}(u)z(u)\right) = 1. 
\end{align*}
Therefore 
\[\lim_{s\to b}w'(s) = \lim_{s\to 0} z'(u) = \frac{1}{1+h_1}.
\]

Now, we check 2. We asume there exist $M>0,$ $s_{1}\in (s_{0},b)$ such that for every $s\in
\lbrack s_{0},b)$ we know $w(s)\geq M.$ There exists $s_{2}\in [s_{1},b)$  such that $h(s)\geq \dfrac{2}{M}$ for every $s\in[s_{2},b)$. Therefore, $1-h(s)w(s)\leq -1.$ By \eqref{problem-ext} we obtain $w'(s)<0$ for every $s\in \lbrack s_{2},b)$. By using $w(s)\geq M$ and $w'(s)<0$
for every $s\in \lbrack s_{2},b)$, there exists $\lim\limits_{s\rightarrow
b}w(s)=w_{1}\geq M.$ Now, by \eqref{problem-ext}, we calculate the limit
\[
\lim\limits_{s\rightarrow b}w'(s)=\lim\limits_{s\rightarrow b}\left[
(1+w^{2}(s))(1-h(s)w(s))\right] =-\infty .
\]

Finally, item 3. We assume there exist $M<0,$ $s_{1}\in (s_{0},b)$ such that for every $s\in
\lbrack s_{0},b)$ we know $w(s)\leq M.$ We know that $h(s)>0$ and $w(s)<0.$
Therefore, by \eqref{problem-ext}, we obtain $w'(s)>0$ for every $s\in
\lbrack s_{1},b)$, namely $w(s)$ is increasing. By using $w(s)\leq M$ and $w'(s)>0$ for every $s\in \lbrack s_{1},b)\ $there exists $\lim\limits_{s\rightarrow b}w(s)=w_{1}\leq M<0.$ Now, by \eqref{problem-ext}, we calculate the limit
\[
\lim\limits_{s\rightarrow b}w'(s)=\lim\limits_{s\rightarrow b}\left[(1+w^{2}(s))(1-h(s)w(s))\right] =+\infty . 
\]
And this completes the proof.
\end{proof}

The case $\varepsilon=\tilde\varepsilon=-1$ can be studied in a similar way. All ideas are already explained, so its proof is left to the reader.

\begin{proposition} \label{extending-1-1}
Given $h:[s_0,b) \rightarrow \mathbb{R}$ , $h\in C^{1}[s_{0},b)$ and $b\leq +\infty$ such that $h(s)<0$. 
\begin{enumerate}
\item For each $w_{0}\in 
\mathbb{R}$, the boundary value problem 
\begin{equation}\label{problem-ext2}
w'(s)=-(1+w^{2}(s))(1-h(s)w(s)), \quad w(s_{0}) =w_{0}, 
\end{equation}
has a unique $C^{2}$-solution $w$ on $[s_{0},b)$.
\item If $b=+\infty$, then $\lim\limits_{s\to +\infty} h(s)w(s)=1$.
\item Assume $b<+\infty$ and there exist the limits 
$\lim\limits_{s\to b}h(s)=-\infty$ and $\lim\limits_{s\to b} \frac{h'(s)}{h^2(s)}= h_1\in (-\infty,0]$. Then, for certain $w_0\in\R$, there exist the limits $\lim\limits_{s\to b} w(s)=0$ and $\lim\limits_{s\to b} w'(s)=\frac{1}{1+h_1}$.  

\item Assume $b<+\infty$ and there exists $\lim\limits_{s\rightarrow
b}h(s)=-\infty$.  If for some $M>0$ and $s_{1}\in \left[ s_{0},b\right)$, it holds $w(s)\geq M$ for every $s\geq s_{1}$,  then there exist $\lim\limits_{s\rightarrow b}w(s)=w_{1}\geq M\ $and $\lim\limits_{s\rightarrow b}w'(s)=-\infty$.

\item Assume $b<+\infty$ and there exist $M<0,$ $s_{1}\in \left[
s_{0},b\right) $ such that for every $s\geq s_{1},$ $w(s)\leq M,$ then there
exist $\lim\limits_{s\rightarrow b}w(s)=w_{1}\leq M\ $and $\lim\limits_{s\rightarrow b}w'(s)=+\infty$.
\end{enumerate}
\end{proposition}

\section{Constructing Translating Solitons}
\label{singularities}

Our next target is to show the existence of graphical translating solitons which are invariant by the action of a Lie group by isometries, under additional conditions. One simple case is the foliation of the Euclidean plane $\R^2$ by circles centered at the origin, where the Lie group is $SO(2)$, but the origin has to be removed in order to obtain smooth maps. In this case, function $h(s)=1/s$, because the geodesic curvature of the circles approaches infinity as the radius tends to zero. 

Consider $(M,g)$ a semi-Riemannian manifold, and  $\Sigma$  a Lie group acting on $M$ by isometries. We  obtain a foliation of $M$  by the orbits of the action of $\Sigma$. Assume (1) there is exactly one (singular) orbit $O$ which is a submanifold, but $0\leq \dim O<\dim M$, and (2) there exists a smooth map $\Phi:(\Sigma/K)\times[0,r)rightarrow M$, for some subgroup $K$, carrying each  $(\Sigma/K)\times\{s\}$, $s\in [0,r)$, into one or the orbits. In other words, we are assuming $O=\Phi((\Sigma/K)\times\{0\})$, and $\Phi:(\Sigma/K)\times (0,r)\rightarrow M\setminus O$ is a diffeomorphism. It is possible to reparametrize it to immediately obtain the projection $\pi:M\setminus O\rightarrow (0,b)$. When $\nabla\pi$ is never light-like, we can recompute $\pi$ to  obtain a semi-Riemannian submersion, \cite{ortega}. That is to say, for each $s\in (0,b)$, $\pi^{-1}\{s\}$ is one orbit, which is a hypersurface of constant mean curvature $h(s)$ because $\Sigma$ acts by isometries, and $\tilde{\varepsilon}=\|\nabla \pi\|^2=\pm 1$. Note that, $\nabla\pi$ is a unit normal vector field along each non-singular orbit. 
\begin{theorem} \label{anterior} 
Under the conditions of this section, assume that the map $q:(0,r)\rightarrow\R$, $q(s)=1/h(s)$, can be extended to $q\in C^1[0,r)$ and satisfies $q(0)=0$, $\tilde{\varepsilon}q'(0)\geq 0$. Then, there exists a smooth map $f:[0,\delta)\rightarrow \R$ such that it induces a graphical translating soliton 
\[\Gamma:\Sigma\times [0,\delta)\rightarrow (M\times\R,\bar{g}=g+\varepsilon dt^2), \quad  \Gamma(\sigma,s)=\big(\Phi(\sigma,s),f(s)\big),\]
whose unit upward normal $\nu$ satisfying $\bar{g}(\nu,\nu)=\tilde{\varepsilon}$. 

In addition, if $\varepsilon=\tilde{\varepsilon}$ and $\varepsilon h(s)>0$ for any $s\in(0,r)$, the translating soliton can be smoothly extended to $\Gamma:M\rightarrow  M\times\R.$
\end{theorem}
\begin{proof} By Theorem \ref{solution}, we just need to define $f(s)=f_1+\int_{0}^s w(u)du$,  $f_1\in\R$ being an integration constant.Then, we just need to use Algorithm \ref{method}.  The boundary condition $f'(0)=0$ is important to ensure the smoothness of the process. 

To extend the translating soliton, we just make use of Section \ref{extensions}.
\end{proof}

\section{Examples}\label{ejemplos}

\begin{example}\normalfont 
In $\mathbb{R}^{n}$, with the standard flat metric $g_{0}$, we consider the
Poincare's Half hyperplane model of $\mathbb{H}^{n}$, namely 
\[
\mathbb{H}^{n}=\left\{ (x_{1},x_{2},...,x_{n})\in \mathbb{R}^{n}\left\vert
{}\right. x_{n}>0\right\} ,\text{ \ }g=\frac{1}{x_{n}^2}g_{0}.
\]
Let $\Sigma$ be the Lie group $\Sigma=(\mathbb{R}^{n-1},+)$ acting by isometries on 
$\mathbb{H}^{n}$ as usual, namely 
\[
\Sigma\times \mathbb{H}^{n} \rightarrow  \mathbb{H}^{n}, \quad 
(w,p)  \rightarrow  (p_{1}+w_{1},...,p_{n-1}+w_{n-1},p_{n})
\]
where $w=(w_{1},w_{2},...,w_{n-1})$ and $p=(p_{1},p_{2},...,p_{n})$, respectively. Note that the orbits are the well-known \textit{horospheres}.

We define the projection map, with its usual properties:
\[
\bar\tau :\mathbb{H}^{n}\rightarrow \mathbb{R},\tau
(x_{1},x_{2},...,x_{n})=\mathrm{ln}(x_{n}). 
\]
Consider two local frames $(\partial _{x_{1}},\partial _{x_{2}},...,\partial
_{x_{n}})$ and $(E_{1}=x_{n}\partial _{x_{1}},E_{2}=x_{n}\partial
_{x_{2}},...,E_{n}=x_{n}\partial _{x_{n}})$ of $T\mathbb{H}^{n}$.  A straightforward computation shows
\begin{equation}
\nabla \bar{\tau} = E_n, \quad \mathrm{div}(\nabla \bar{\tau}) =  -n+1.
\end{equation}
We arrive to the following initial value problem,
\begin{equation}
f^{''}(s)=(1+(f'(s))^{2})(1+(n-1)f'(s)), \quad 
f'(s_{0})=f_{0},\quad f(s_{0})=f_{1},  \label{1}
\end{equation}
where $s_0, f_0, f_1\in\mathbb{R}$. By the easy change $f'(s)=w(s)$, we transform this problem in
\begin{equation}
w^{\prime}(s)=(1+w^{2}(s))(1+(n-1)w(s)),\quad w(s_0)=f_0.  \label{2}
\end{equation}
The classical change of variable $t=w(s)$ allows to compute a first integral, by the expression 
$F'(t)=\dfrac{1}{(1+(n-1)t)(1+t^{2})}$, so that 
\begin{gather*}
F:\mathbb{R}\backslash \left\{ -1/(n-1)\right\} \rightarrow \mathbb{R},\\
F(t)=\frac{1}{1+(n-1)^{2}}\left[ (n-1)\ln \left( \frac{\left\vert
	1+(n-1)t\right\vert }{\sqrt{1+t^{2}}}\right) +\arctan t\right] +C_{0},
\end{gather*}
for some integration constant $C_0\in\R$.  From here we obtain 3 cases.

\noindent\underline{Case 1:} $f_0=-1/(n-1)$. Then, the function $w(s)=\frac{-1}{n-1}$ is a constant solution to \eqref{2}. Thus, $f(s)=f_1-s/(n-1)$ is a solution to \eqref{1}. \newline

\noindent \underline{Case 2}: $f_{0}>-1/(n-1)$. We restrict $F$, namely $F_{1}:\left( \frac{-1}{n-1},+\infty \right) \rightarrow \mathbb{R}$. In this case, $F'>0$, so that $F$ is injective. To compute its image, we see 
\begin{eqnarray*}
\lim_{t\rightarrow +\infty }F_{1}(t) &=&\lim_{t\rightarrow +\infty }\left[ 
\frac{1}{1+(n-1)^{2}}\left[ (n-1)\ln \left( \frac{\left\vert
	1+(n-1)t\right\vert }{\sqrt{1+t^{2}}}\right) +\arctan t\right] +C_{0}\right] 
\notag \\
&=&\left( \frac{1}{1+(n-1)^{2}}\right) \left( (n-1)\mathrm{ln}(n-1)+\frac{\pi }{2}\right) +C_{0}=:K_{0},  \\ 
\lim_{t\rightarrow \frac{-1}{n-1}^{+}}F_{1}(t) &=&\ \lim_{t\rightarrow \frac{-1}{n-1}^{+}}\left[ \frac{1}{1+(n-1)^{2}}(n-1)\ln \left( \frac{\left\vert
	1+(n-1)t\right\vert }{\sqrt{1+t^{2}}}+\arctan t\right) +C_{0}\right]
=-\infty .  \notag
\end{eqnarray*}
We obtain that $F_{1}:\left( \frac{-1}{n-1},+\infty \right) \rightarrow
(-\infty ,K_{0})$ is bijective, and there exists its inverse function 
\begin{equation*}
F_{1}^{-1}:(-\infty ,K_{0})\rightarrow \left( \frac{-1}{n-1},+\infty \right).
\end{equation*}
Now, we recover $w(s)=F_{1}^{-1}(s)$, $w(s_{0})=F_{1}^{-1}(s_{0})=w_{0}>\frac{-1}{n-1}$, and $\underset{s\rightarrow -\infty }{\lim }w(s)=-1/(n-1)$, $\underset{s\rightarrow K_{0}}{\lim }w(s)=+\infty $. Finally, 
\begin{equation*}
f:(-\infty ,K_{0})\rightarrow \mathbb{R},\quad
f(s)=f_{1}+\int_{s_{0}}^{s}w(u)du.
\end{equation*}
Then, we obtain $\lim\lim\limits_{s\rightarrow -\infty }{\lim }f(s)=-\infty$ and $\lim\limits_{s\rightarrow K_{0}}{\lim }f(s)=+\infty$. 
Thus, function $f$ has a finite time blow up.\newline

\noindent \textit{Case 3:} $f_0<-1/(n-1)$. As in the previous case, we
restrict $F$ to $F_2:(-\infty,-1/(n-1))\rightarrow\mathbb{R}$ and compute
its image. Indeed,

\begin{eqnarray*}
\lim_{t\rightarrow -\infty }F_{2}(t) &=&\lim_{t\rightarrow -\infty }\left[ 
\frac{1}{1+(n-1)^{2}}\left[ (n-1)\ln \left( \frac{\left\vert
	1+(n-1)t\right\vert }{\sqrt{1+t^{2}}}\right) +\arctan t\right] +C_{0}\right] 
\notag \\
&=&\left( \frac{1}{1+(n-1)^{2}}\right) \left( (n-1)\mathrm{ln}(n-1)-\frac{\pi }{2}\right) +C_{0}=:K_{1},  \\ 
\lim_{t\rightarrow \frac{-1}{n-1}^{-}}F_{2}(t) &=&\ \lim_{t\rightarrow \frac{-1}{n-1}^{-}}\left[ \frac{1}{1+(n-1)^{2}}\left[ (n-1)\ln \left( \frac{\left\vert 1+(n-1)t\right\vert }{\sqrt{1+t^{2}}}\right) +\arctan t\right]
+C_{0}\right] \\ &=&-\infty .  \notag
\end{eqnarray*}
For $t<\frac{-1}{n-1}$, then $F_{2}'(t)<0$, so that we obtain the
bijection $F_{2}:\left( -\infty ,\frac{-1}{n-1}\right) \rightarrow \left(
-\infty ,K_{1}\right) $, that is to say, 
\begin{equation*}
F_{2}^{-1}:(-\infty ,K_{1})\rightarrow \left( -\infty ,-1/(n-1)\right) .
\end{equation*}
Now, we recover $w(s)=F_{2}^{-1}(s)$, $w(s_{0})=F_{2}^{-1}(s_{0})=f_{0}<
\frac{-1}{n-1}$, and $\underset{s\rightarrow -\infty }{\lim }w(s)=-1/(n-1)$, 
$\underset{s\rightarrow K_{1}}{\lim }w(s)=+\infty $. Finally, 
\begin{equation*}
f:(-\infty ,K_{1})\rightarrow \mathbb{R},\quad
f(s)=f_{1}+\int_{s_{0}}^{s}w(u)du.
\end{equation*}
Then, we obtain $\lim\lim\limits_{s\rightarrow -\infty}{\lim }f(s)=-\infty$ and $\lim\limits_{s\rightarrow K_{1}}{\lim }f(s)=+\infty$. 
Therefore, function $f$ has a finite time blow up.\newline

Next, for each case, we resort to Algorithm \ref{method} to obtain our
translating solitons. Finally, this example shows that the condition $h>0$ cannot be removed in Proposition \ref{extending+1+1}.
\end{example}

\begin{example}\normalfont
In $\mathbb{R}^{n+1}$, $n\geq 1$, with its standard flat metric $g,$
consider a round $n-$sphere of radius $1$ centered at $0,$ namely $\mathbb{S}^{n}.$ As usual, we identify the tangent space at $x\in \mathbb{S}^{n}$,  
\[
T_{x}\mathbb{S}^{n}=\left\{ X=\left( X_{1},...,X_{n+1}\right) \in \mathbb{R}^{n+1}:g(X,x)=0\right\} .
\]
Now, the Lie group $O(n-1)$ acts by isometries on $\mathbb{S}^{n}$ as usual: 
\[
O(n-1)\times \mathbb{S}^{n}\rightarrow \mathbb{S}^{n},\text{ }\left(
A,x\right) \rightarrow A.x=\left( 
\begin{array}{cc} A & 0 \\  0 & 1 \end{array}
\right) x=\left( 
\begin{array}{c}
A(x_{1},...,x_{n})^{t} \\ x_{n+1}
\end{array} \right) 
\]
We restrict our study to $M=\mathbb{S}^{n}\backslash \{N,S\}$, i. e., we remove the North and South Poles. In this way, the space of orbits can be identified by the following projection map
\[
\tau:M\rightarrow (-\pi/2,\pi/2),\quad \tau(x)=-\arcsin (x_{n+1}).
\]
Then, given $\xi=(0,\ldots,0,1)$, simple computations show $x_{n+1}=g(x,\xi)$ and 
\[\nabla \tau (x)=- \frac{\xi -g(x,\xi )x}{\sqrt{1-x_{n+1}^2}}=-\frac{\xi -g(x,\xi )x}{\cos(\tau(x))},\ x\in \mathbb{S}^{n}, \quad  \|\nabla\tau\|=1, \quad \mathrm{div}(\nabla\tau) = (n-1)\tan(\tau).
\]
We obtained that $h(s)=(1-n)\tan(s)$. Thus, we consider the following differential equation:
\begin{equation}\label{sphere}
f''(s)=\big(1+f'(s)^2\big)\big(1-(n-1)\tan(s)f'(s)\big), 
\end{equation}
which we reduce in a first step to ($w=f'$),
\begin{equation}\label{esfera}
w'(s)=\big(1+w^2(s)\big)\big(1-(n-1)\tan(s)w(s)\big). 
\end{equation}
Needless to say, for each $w_o\in\R$, there exists a solution $w:(-\delta,\delta)\rightarrow\R$ such that $w(0)=w_o$. Since $h>0$ on $(0,\pi/2)$, by Proposition \ref{extending+1+1}, we can extend to $w:(-\delta,\pi/2)\rightarrow\R$. Now, by taking $z:(-\pi/2,\delta)\rightarrow\R$, $z(u)=-w(-u)$, it is clear that $z$ is another solution to \eqref{esfera}. By Proposition \ref{extending+1+1}, we can extend $z:(-\pi/2,\pi/2)\rightarrow\R$. This means that each solution to \eqref{esfera} can be globally defined $w:(-\pi/2,\pi/2)\rightarrow\R$. Clearly, for each $f_0\in\R$, we construct a solution $f(s)=\int w(x)dx +f_0$, $f:(-\pi/2,\pi/2)\rightarrow\R$. Now, by using Algorithm \ref{method}, given a solution $f$, we obtain a translating soliton defined on the sphere except two points, namely $\mathbb{S}^n\backslash\{N,S\}$.

Moreover, a simple computation shows
\[ 
\lim\limits_{s\to\pi/2} \frac{h'(s)}{h^2(s)} = \lim\limits_{s\to\pi/2} \frac{1}{(n-1)\sin^2(s)}= \lim\limits_{s\to -\pi/2} \frac{h'(s)}{h^2(s)} = \frac{1}{n-1}
>0.\]
By Proposition \ref{extending+1+1}, there exist two solutions $\bar{w}$ and $\tilde{w}$ that satisfy the conditions   $\lim\limits_{s\to\pi/2}\bar{w}(s)=0$ and $\lim\limits_{s\to -\pi/2}\tilde{w}(s)=0$. We take $\bar{f}=\int \bar{w}$ and $\tilde{f}=\int\tilde{w}$, and use Algorithm \ref{method}. Then, these translating solitons will admit a  tangent plane at points $N$ or $S$, i.~e., they will be smooth. 
Problem is, we have not been able to show if these two translating solitons coincide. 
\end{example}

\begin{example} \normalfont 
In the standard Euclidean Space $\R^3$, consider a $C^{\infty}$ plane curve $\alpha:(0,+\infty)\rightarrow\R^3$, $\alpha(s)=(x(s),0,z(s))$, $x(s)>0$ for any $s>0$, which is arclength, and satisfies $\lim\limits_{s\to 0}\alpha(s)=0$, $\lim\limits_{s\to 0}x'(s)=x_0>0$. We construct the revolution surface $S$ parametrized by 
\[ X:\mathbb{S}^1\times (0,+\infty)\rightarrow \R^3, \
X(\theta,s)=(\cos(\theta)x(s),\sin(\theta)y(s),z(s)). 
\]
This is a smooth surface foliated by cirles of radius $1/x(s)$, or rather, invariant by the Lie group $\mathbb{S}^1$ acting by isometries. In this case, consider the map 
\[ h:(0,\infty)\rightarrow \R, \ h(s)=1/x(s).
\]
Since $q(s):=1/h(s)=x(s)$ can be smoothly $C^1$-extended to $q:[0,+\infty)\rightarrow\R$, with $q'(0)=x_0>0$, by our previous results, we can construct a $\mathbb{S}^1$-invariant translating soliton $S\rightarrow S\times\R\subset\R^4$. Note that function $h$ is here totally arbitrary.  
\end{example}

\section*{Acknowledgements}

M.~Ortega has been partially financed by the Spanish Ministry of Economy and Competitiveness and European Regional Development Fund (ERDF), project  MTM2016-78807-C2-1-P.

\end{document}